\documentclass[11pt]{amsart}
\usepackage{amssymb}
\usepackage{amsfonts}
\usepackage{epsfig}
\usepackage{amsmath}
\usepackage{graphicx}
\usepackage{caption}
\usepackage{subcaption}
\usepackage{wrapfig}
\usepackage{mathtools}

\newtheorem{thm}{Theorem}

\newtheorem{lemma}[thm]{Lemma}
\newtheorem{corollary}[thm]{Corollary}

\theoremstyle{definition}

\begin{document}

\title{Waist Size for Cusps in  Hyperbolic\\ 3-Manifolds II}

\begin{abstract}
The waist size of a cusp in an orientable hyperbolic 3-manifold is the length of the shortest nontrivial curve generated by a parabolic isometry in the maximal cusp boundary. Previously, it was shown that the smallest possible waist size, which is 1, is realized only by the cusp in the figure-eight knot complement.  In this paper, it is proved that the next two smallest waist sizes are realized uniquely for the cusps in the $5_2$ knot complement and the manifold obtained by (2,1)-surgery on the Whitehead link. One  application is an improvement on the universal upper bound for the length of an unknotting tunnel in a 2-cusped hyperbolic 3-manifold. 
\end{abstract}

\author[Colin Adams]{Colin Adams}
\address{Department of Mathematics and Statistics, Bronfman Science Center, Williams College, Williamstown, MA 01267}
\email{Colin.C.Adams@williams.edu}

\date{\today}

\maketitle

\section{Introduction}

In \cite{Adams3}, the waist size of a cusp in a hyperbolic 3-manifold was introduced. It is defined to be the shortest nontrivial curve in the boundary of a maximal cusp in the manifold that is generated by a parabolic isometry. It was shown there that all cusps in hyperbolic 3-manifolds have waist size at least 1 and that the unique cusp in a hyperbolic 3-manifold with waist size 1 is the cusp in the figure-eight knot complement. In this paper, we extend those results to determine the initial segment of manifolds with small waist size. Specifically, we show that the $5_2$ knot and the manifold obtained by (2,1)-surgery on the Whitehead link are the only manifolds other than the figure-eight knot that have waist size at most $\sqrt[4]{2}$. 

There have been a variety of applications of these results based on an unpublished preprint of this article that was previously circulated. See \cite{CFP}, \cite{FKP1}, \cite{FKP2}, \cite{FKP3}, \cite{FS}, \cite{Ich}, \cite{LP1}, \cite{LP2} and \cite{PV}.  These results are relevant to the number of exceptional surgeries since the 6-Theorem of \cite{Agol} and \cite{Lack} shows that a Dehn filling on a cusped hyperbolic 3-manifold that yields a non-hyperbolic manifold must have length at most 6. Hence a lower bound on the length of closed curves in the cusp can improve the lower bound on the number of exceptional fillings possible.  

An additional application is given in Section 4,  where we use the results to obtain an improvement on the universal upper bound for the length of an unknotting tunnel in a 2-cusped hyperbolic 3-manifold. 

Note that although for most knots, the waist size corresponds to the meridian length in a maximal cusp, there are counterexamples to that being true in general. An infinite class of such examples is given in \cite{BH}.  Also note there are many examples of infinite families of manifolds with equal waist sizes. See \cite{Adams3} for details.

 We begin with some background. Given a hyperbolic 3-manifold $M$ and a maximal cusp $C$, the cusp lifts to an infinite set of horoballs in the upper-half-space model of hyperbolic 3-space with disjoint interiors and some points of tangency on their boundaries. This is called the horoball packing corresponding to the cusp in the manifold.  Each horoball touches the boundary of hyperbolic space at a single point.  This point is called the center of the horoball. For convenience, we  choose the point at $\{\infty \}$ to be the center of a horoball denoted $H_{\infty}$ that covers the maximal cusp. We will normalize so that the boundary of this horoball is a plane at height $z=1$. Since the maximal cusp touches itself, every horoball must be tangent to other horoballs. In particular, there must be horoballs tangent to $H_{\infty}$. These horoballs, which have Euclidean diameter 1, are called full-sized horoballs.  When looking down at the $xy$-plane in the upper-half-space model, we will see circles corresponding to the horoballs.  In the case that the maximal cusp has finite volume, there is a $\mathbb{Z} \times \mathbb{Z}$ subgroup of parabolic isometries fixing $\{\infty \}$ in the manifold group, which has a parallelogram in the $xy$-plane for its fundamental domain. We can choose the vertices of the parallelogram to occur at the center of four full-sized balls, all of which are identified by the cusp subgroup. The  parallelogram, together with the circles corresponding to the horoballs is called the cusp diagram.  We choose the parallelogram so that one of its edges is along the shortest translation of a parabolic isometry in the cusp subgroup.  Thus, the Euclidean length of that edge in this normalized model is exactly the waist size. In what follows, that length is denoted w. We will denote the corresponding parabolic translation by $P$.
        
     In the case of a cusp with infinite volume, there is a single parabolic isometry fixing $\{\infty \}$, again denoted $P$, which generates the entire subgroup of the fundamental group fixing $\{\infty \}$. A fundamental domain for that subgroup is a strip of the $xy$-plane between two parallel lines. In either the finite volume or infinite volume case, let $\Gamma$ denote the fundamental group of the manifold realized as a discrete group of isometries of hyperbolic 3-space and let $\Gamma_{\infty}$ denote the subgroup of isometries in the fundamental group that fix $\{\infty \}$.
     
       Note that when we refer to the distance between the centers of horoballs, this is a Euclidean distance in the $xy$-plane that forms the boundary of the upper-half-space model of $\mathbb{H}^3$. The radius or diameter of a horoball is also a Euclidean distance where we have normalized so that the full-sized horoballs have diameter 1.  Any length measured in the horosphere given by the plane $z=1$ is simultaneously a Euclidean and hyperbolic length. There follows a list of relevant lemmas, proofs for which appear in \cite{Adams3}.

\begin{lemma}( 2.6 of \cite{Adams3}). Up to the action of $\Gamma_{\infty}$, every horoball other than $H_{inbfty}$ has a horoball of the same diameter paired to it, which is called its associated horoball.
\end{lemma}

\begin{lemma}(2.7 of \cite{Adams3}). If there exists a horoball $H_1$ of diameter $b$ with center a distance $c$ from the center of a full-sized horoball $H_2$ in the horoball packing, then there exists a horoball $H_3$ of diameter $\frac{b}{c^2}$, with center a distance $\frac{1}{c}$ from the full-sized ball associated to $H_2$.
\end{lemma}

\begin{lemma}( 2.8 of \cite{Adams3}). Given a horoball $H$ of diameter $k$ in the horoball packing, both it and its associated horoball $H'$ have a pair of horoballs on either side of them with centers a distance $\frac{k}{w}$ from the centers of $H$ and $H'$, and with diameter $(\frac{k}{w})^2$. The line segment between the centers of the pair associated to $H$ passes through the center of $H$ and is parallel to the line segment between the centers of the pair associated to H' which passes through the center of $H'$.
\end{lemma}

\begin{lemma} \label{biglemma}(2.9 of \cite{Adams3}).  Given a horoball packing corresponding to a maximal cusp, up to the action of the cusp subgroup:

\begin{enumerate}

\item There exist at least two distinct full-sized horoballs, and full-sized horoballs always come in pairs.

\item Each full-sized ball has at least two balls tangent to it, called $1/w$-balls, each of diameter $\frac{1}{w^2}$, such that the full-sized ball is centered at the midpoint of a line segment of length $\frac{2}{w}$ in the $xy$-plane and such that the endpoints of the line segment are the centers of these two $1/w$-balls. The line segments corresponding to the $1/w$-balls for a pair of associated full-sized balls are parallel.

\item Each $1/w$-ball has a ball tangent to it, called a $1/e$-ball. Its center is a distance $\frac{1}{e}$ from the full-sized ball that the $1/w$-ball is tangent to and a distance $\frac{1}{w^2e}$ from the center of the $1/w$-ball.

\end{enumerate}

\end{lemma}

\section{Additional Horoballs}

      We now build on Lemma \ref{biglemma} and demonstrate that the pattern of horoballs that we already know to exist forces the existence of a variety of additional horoballs.

\begin{lemma}($1/w^3$-balls) For the cusp diagram of any maximal cusp in a hyperbolic manifold M, there exist two horoballs to either side of each $1/w$-ball, with centers a distance $\frac{1}{w^3}$ from the center of the $1/w$-ball and diameters $\frac{1}{w^6}$. One of them touches the $1/e$-ball that touches the $1/w$-ball.

\end{lemma}

\begin{proof}There exists an isometry $J$ of the horoball pattern that takes a given $1/w$-ball to $H_\infty$ and $H_\infty$ to the horoball associated to the $1/w$-ball. By Lemma 2, the translates by $P$ and $P^{-1}$ of the $1/w$-ball will be sent by $J$ to balls of radius $\frac{1}{w^6}$, each centered a distance $1/w$ from the center of the associated ball, on either side of it. Similarly the translates of the associated ball by $P$ and $P^{-1}$ are sent by $J^{-1}$ to balls of radius $\frac{1}{w^6}$, each centered a distance $1/w$ from the center of the $1/w$-ball.
\end{proof}

     Up to this point, we have seen that there exist full-sized balls, $1/w$-balls, $1/e$-balls and $1/w^3$-balls. Particular values for $w$ and $e$ determine the centers and diameters of all of these balls. As in Figure \ref{distances}, we define the Euclidean distances $e, v, y, m, p, s$ and $k$ between the centers of pairs of horoballs that are known to exist. Note that if any one of these distances equals 0, the two horoballs centered at the endpoints of the corresponding line segment are identical.

\begin{figure}{}
\begin{center}
\includegraphics[scale= .8]{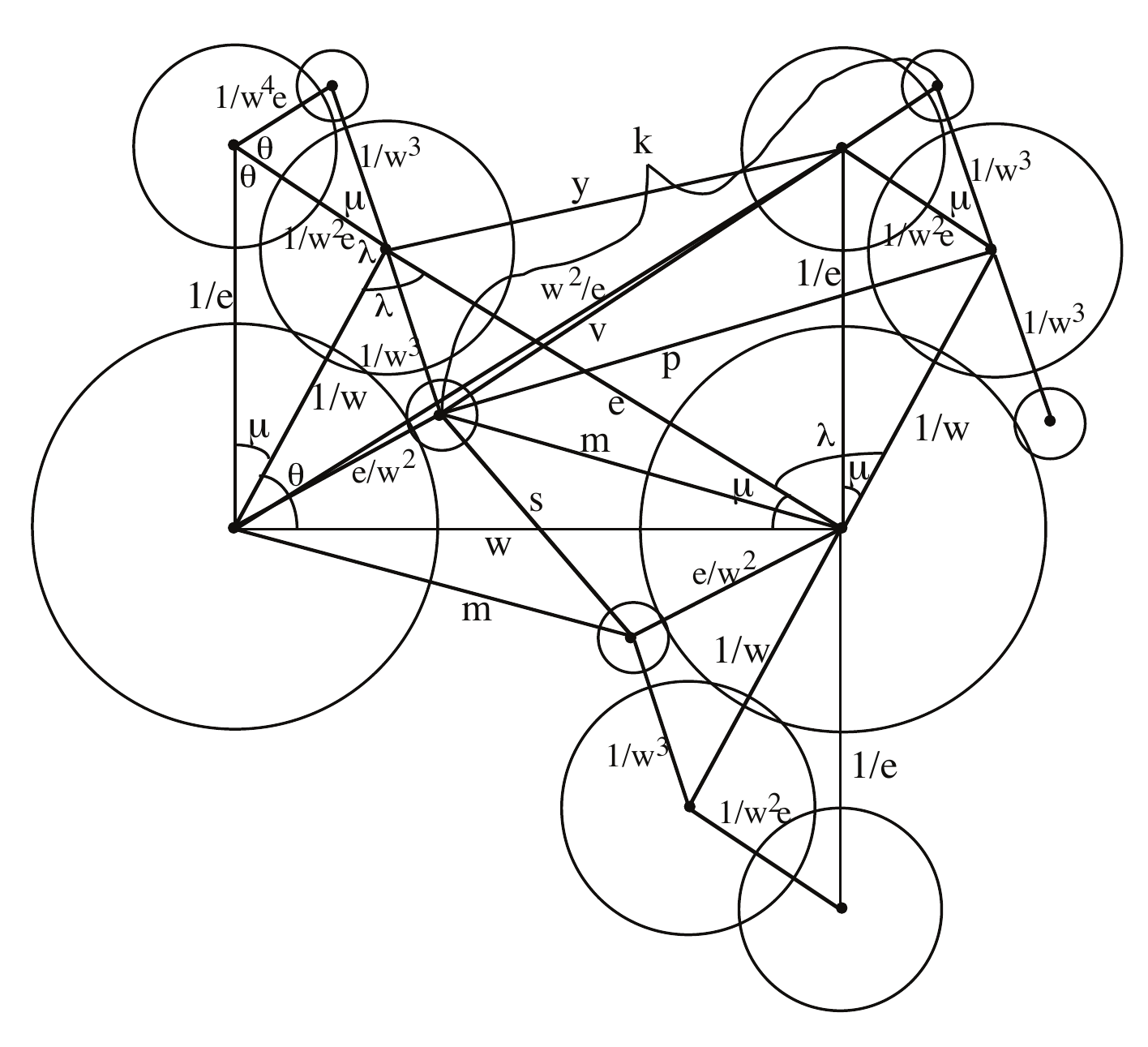}
\caption{A choice of $w$ and $e$ determine all horoballs in this figure.}
\label{distances}
\end{center}
\end{figure}

     On the other hand, if one of the defined distances is not 0, then it must be large enough to ensure that the corresponding pair of  horoballs do not overlap in their interiors. Thus we generate a set of inequalities in $w$ and $e$ that must be satisfied. In the following lemmas, we gather together a set of these inequalities that will prove useful.

\begin{lemma}(Upper $e$ inequality) For any cusp in a hyperbolic manifold, 
$w^4-e^2w^2+1 \geq 0$.
\end{lemma}

\begin{proof} By definition, $e$ is the shortest distance between a $1/w$-ball and the translation by $P$ or $P^{-1}$ of the full-sized ball that it touches. Hence, $\theta$ is at most $\pi/2$. The law of cosines yields the inequality.
\end{proof}

\begin{lemma}For each labelled line segment between centers of horoballs depicted in Figure \ref{distances} that has non-zero length, the following inequality must hold:

\begin{enumerate}

\item (Lower $e$ inequality) $ew -1 \geq 0$

\item ($v$-inequality) $w^{12} + w^8(-1-e^4) + w^6(e^6-e^2) + w^4(2-2e^4) + w^2(3e^2 )-2 \geq 0$

\item ($y$-inequality) $e^4w^2 + w^6 -e^2w^4 -w^2 -e^2 \geq 0$

\item ($s$-inequality)  $w^{14}- 2e^2w^{12 }+ w^{10}(2e^4-2)  + 2w^6 -1 \geq 0$

\item ($m$-inequality) $e^4w^2 + w^6 - e^2w^4 -w^2 - e^2 \geq 0$

\item ($p$-inequality) $2w^{10} +2w^2 +e^4w^6 -2e^2w^4 -2e^2w^8 -1 \geq 0$

\item ($k$-inequality) $3w^{14}  - 4w^{12}e^2 +2w^{10}e^4 - 4w^8e^2 + 6w^6 -1 \geq 0$

\end{enumerate}

\end{lemma}

\begin{proof} The first inequality follows from the fact that if $e$ is as small as possible but greater than 0, the $1/w$-ball will be tangent to the corresponding full-sized ball, which occurs for $e=1/w$. For the $v$-inequality, if the $1/w^3$-ball is not coincident with the $1/e$-ball, as occurs for $v> 0$, then the two balls of diameters $\frac{1}{w^6}$ and $\frac{1}{w^2e^2}$ respectively must have centers a distance $\frac{1}{w^4e}$ apart. By the law of cosines, we can determine that $v^2 = (\frac{e}{w^2})^2 - 2\cos(\theta - 2 \mu)$ By applying trig identities, we obtain the given inequality.

    It must be the case $y \geq \frac{1}{w^2e}$, if the two horoballs centered at the ends of the line segment corresponding to $y$ are not to overlap in their interiors. Then,  $y^2 = e^2 + \frac{1}{e^2} - 2 \cos (\lambda - \mu)$ and repeated applications of trig identities and the law of cosines gives the desired $y$-inequality.
    
     For the $s$-inequality, we know that $s \geq \frac{1}{w^6}$ if the corresponding two horoballs are not to overlap in their interiors. Then, $(\frac{s}{2})^2 = (\frac{e}{w^2})^2 = (\frac{w}{2})^2 - \frac{e}{w} \cos (\theta - \mu)$, and trig identities together with the law of cosines ultimately yield the $s$-inequality.
     
     The length $m$ must be at least $\frac{1}{w^3}$ if the two horoballs centered at its ends are not to overlap. Then, $m^2 = \frac{1}{w^6} + e^2 - \frac{2e}{w} \cos (\gamma+ \mu)$, and via trig identities, we again obtain the given $m$-inequality.
     
     It must be the case that $p \geq \frac{1}{w^4}$ for the two horoballs corresponding to $p$ to avoid overlapping in their interiors. Then, $p^2 = \frac{1}{w^2} - m^2 - \frac{2m}{w} \cos (\gamma+\lambda)$, and via trig identities, the law of cosines and substitutions, we obtain the given $p$-inequality.
     
     For the $k$-inequality, the two $1/w^3$-balls have centers with coordinates given by $$\big( \frac{1}{w} \cos\theta + \frac{1}{w^3} \sin(2\theta - \frac{\pi}{2}), \frac{1}{w} \sin \theta - \frac{1}{w^3}\cos (2\theta -\frac{\pi}{2})\big)$$  and  $$\big( \frac{1}{w} \cos\theta - \frac{1}{w^3}\sin(2\theta - \frac{\pi}{2}) + d, \frac{1}{w} \sin \theta + \frac{1}{w^3}\cos (2\theta -\frac{\pi}{2})\big).$$ The fact that their centers must have distance apart at least twice their radius implies $k \geq \frac{1}{w^6}$, and yields the desired $k$-inequality.
\end{proof}

     When one of the given lengths is non-zero and the inequality is satisfied, but the length is not too large, the fact that these two horoballs come relatively close to one another implies that when one of the horoballs is sent to $H_{\infty}$ by an isometry, the other horoball is sent to a relatively large horoball in the diagram. For example, if the inequality is an equality, the two horoballs touch and therefore the new horoball that is generated is full-sized. This will be of use in the following section.

\section{Waist Sizes Greater than 1}

\begin{lemma} \label{knot5_2}Let $M$ be a cusped hyperbolic manifold such that $w > 1$ and the $1/w^3$-ball corresponding to a given $1/w$-ball is coincident with the $1/e$-ball corresponding to a translation by $P$ or $P^{-1}$ of the $1/w$-ball. Then the manifold is the $5_2$ knot complement.
\end{lemma}

\begin{proof} Under these assumptions, the diameter of the $1/e$-ball, which is $\frac{1}{w^2e^2}$, equals the diameter of the $1/w^3$-ball, which is $1/w^6$. Thus, $e = w^2$. The diagram for the centers of the one set of full-sized horoballs and adjacent $1/w$-balls, and $1/e$-balls must appear as in Figure \ref{coincidingballs}.
\end{proof}

  \begin{figure}{}
\begin{center}
\includegraphics[scale= .8]{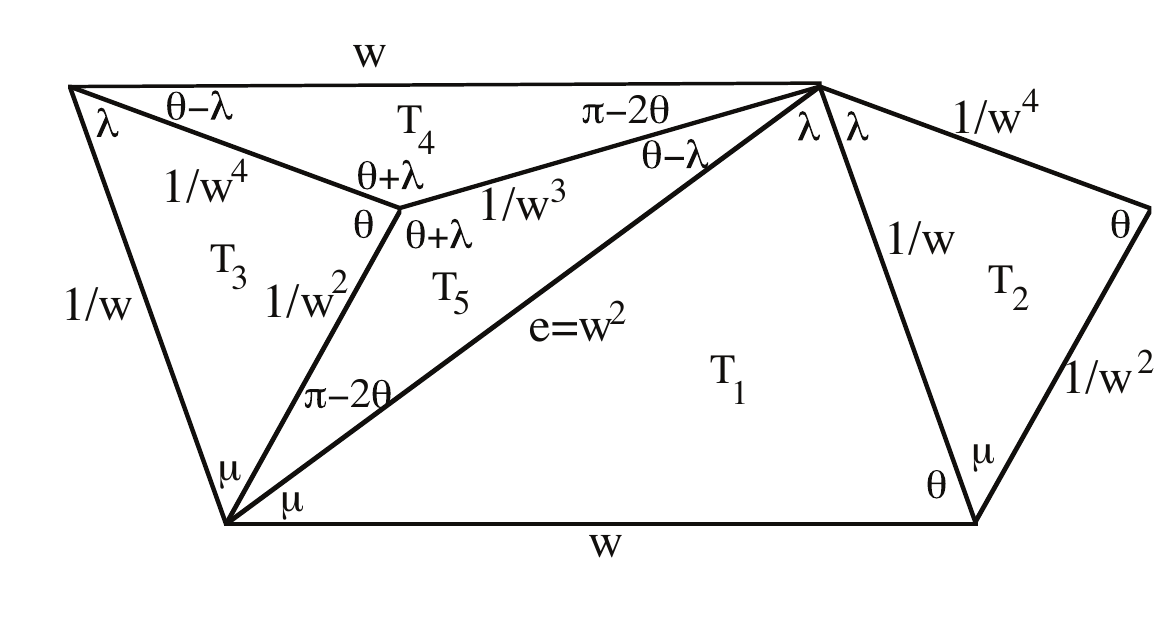}
\caption{Pattern of horoball centers when $1/e$-ball and $1/w^3$-ball coincide.}
\label{coincidingballs}
\end{center}
\end{figure}

    Since the triangles $T_1, T_2$ and $T_3$ are similar, as are the triangles $T_4$ and $T_5$, one can label the angles as in Figure \ref{coincidingballs}.  By the law of cosines we have $w^4 = w^2 + \frac{1}{w^2} - 2 \cos \theta$  and $\frac{1}{w^8} = w^2 + \frac{1}{w^6} - \frac{2}{w^2} \cos (\pi - 2 \theta)$  
Upon solving for w, we find  

$$0 = w^{14} -2w^{12} + 2w^{10} - 2w^8 + 2w^2 - 1 = (w-1)(w+1)(w^3 - w^2 + 1)(w^6 -w^2 -1)$$

The only real root greater than 1 is the root $w = 1.150964\dots$ of  $w^6 -w^2 -1$. This is the waist size of the $5_2$ knot complement. It remains to show that the $5_2$ knot complement is the only manifold with this waist size such that the $1/e$-ball and one of the $1/w^3$-balls coincide.
 
      At this point, two pieces of our cusp diagram must appear as in Figure \ref{cusppieces}, where the labels on the edges are forced on us by the isometries that identify the vertical edge labelled 1 and pointed up to the vertical edge labelled 1 and pointed down (isometry I), as well as by the identifications of 2 up to 2 down, and 3 up to 3 down. Call the southernmost piece the $A$ piece and the northernmost piece the $B$ piece.

  \begin{figure}{}
\begin{center}
\includegraphics[scale= .6]{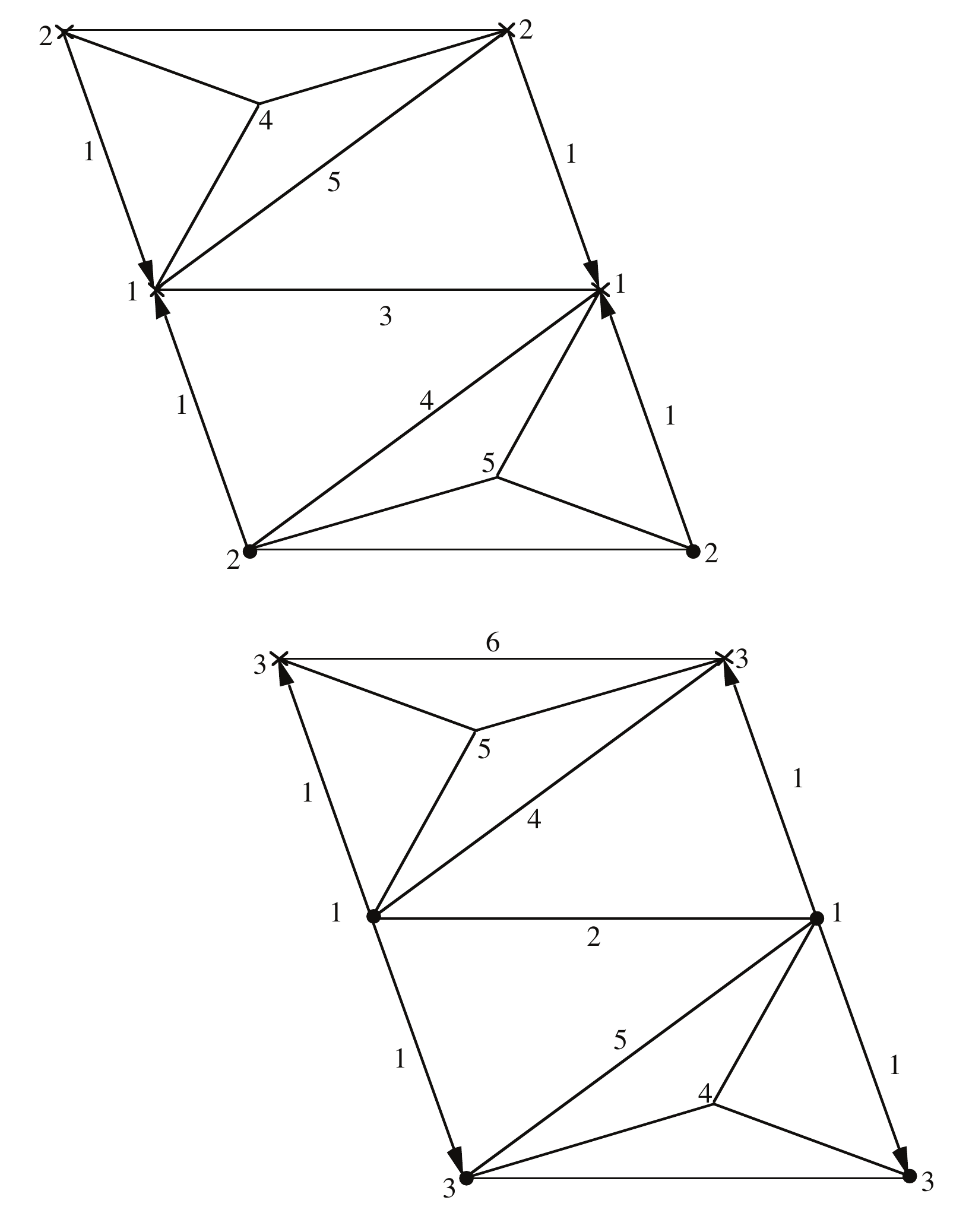}
\caption{Labellings that are forced on us when $w = 1.150964\dots$.}
\label{cusppieces}
\end{center}
\end{figure}

     In particular, notice that the isometry $I$ will identify the ÒhorizontalÓ edge labelled 4 in the $A$ piece with the vertical edge labelled 4 in the $B$ piece. Similarly, $K$ will identify the vertical edge labelled 4 in the $B$ piece with the ÒhorizontalÓ edge labelled 4 in the $B$ piece. Then, $L$ identifies  the vertical edge labelled 4 in the $A$ piece with the ÒhorizontalÓ edge labelled 4 in the $A$ piece.
     
     When the vertical edge labelled 3 up is sent to the vertical edge labelled 3 down, the two edges labelled 6 that shared the bottom endpoint of the vertical edge labelled 3 up will be sent to vertical edges coming out of a pair of $1/w^3$-balls surrounding the $1/w$-ball at the bottom of the vertical edge labelled 3. This forces the label 6 to coincide with the label 4. The fact that there are then two vertical 4 labels, one up and one down, on either side of the downward pointing 3 edge means that the downward pointing 4 label in $B$ must coincide with one of these downward pointing 4 labels. This can only occur if the 2 edge label coincides with the 3 edge label.Thus, the two pieces $A$ and $B$ fit together to give a cusp diagram that exactly coincides with the cusp diagram for the $5_2$ knot, at least up to the horoballs we have so far discussed.  For convenience, we drop the edge labelled 4 (and 6). One checks that the two $1/w^3$-balls touch each other, and therefore that the edge between them is 1. Our picture now appears as in Figure \ref{singlecusp}.

 \begin{figure}{}
\begin{center}
\includegraphics[scale= .6]{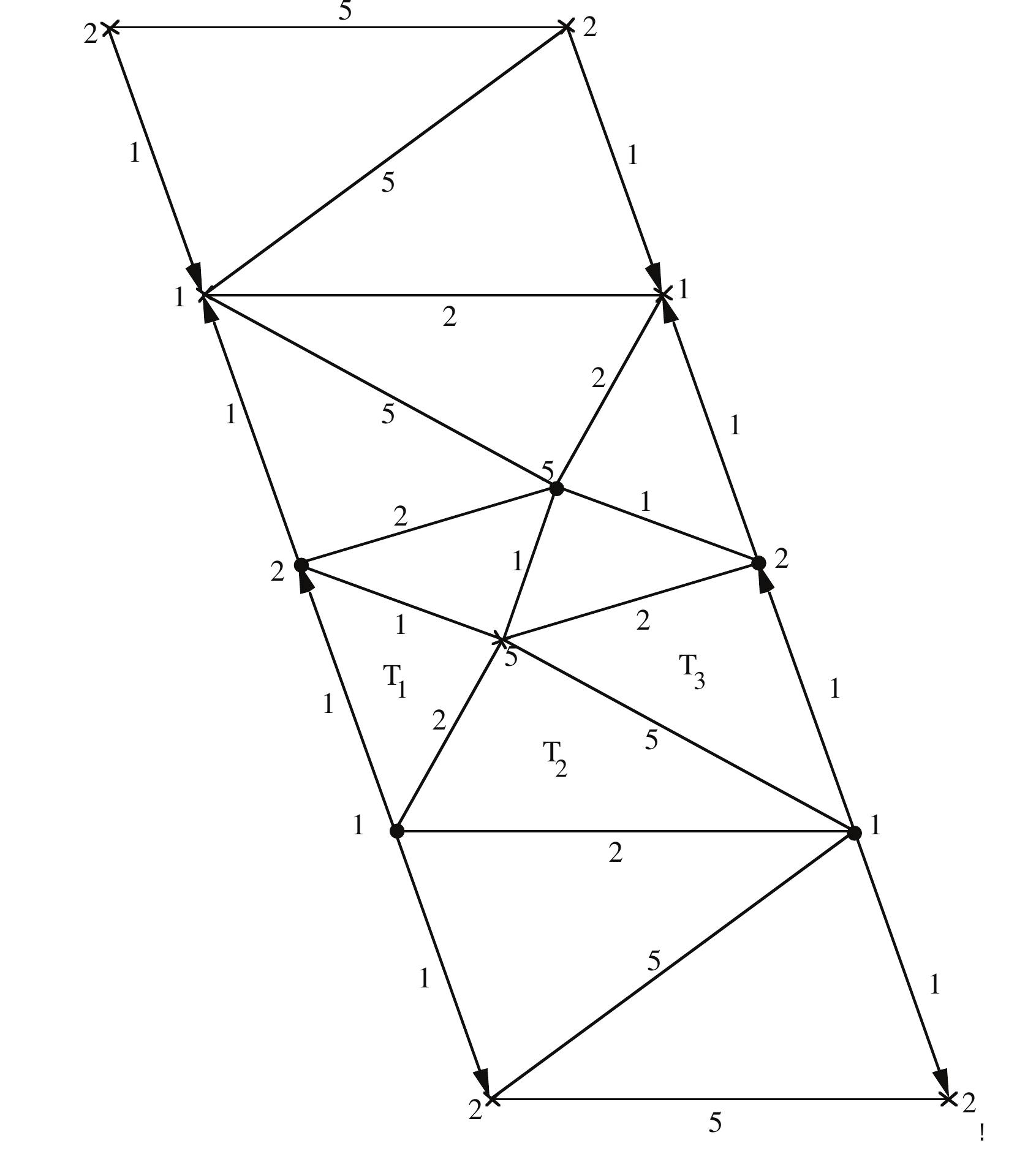}
\caption{Cusp diagram.}
\label{singlecusp}
\end{center}
\end{figure}

Then, all of the vertical ideal tetrahedra labelled in the picture are identified to one of $T_1, T_2$ or $T_3$.  No points in the interiors of these three tetrahedra are identified with one another, as if they were, the tetrahedra themselves would have to be identified in order to respect the horoball packing, forcing more than two vertical edges with the same label, a contradiction. Note that all three of these tetrahedra are isometric with one another.  Since the faces of these tetrahedra are paired with one another, the fundamental domain of the manifold can contain no other material, as it would be disconnected from this part in the quotient.

      Thus any manifold that has the $1/w^3$-ball coincident with the 1/e-ball and waist size $1.150964\dots$  must lift to this picture with three tetrahedra glued to one another along faces in this manner.  The three tetrahedra form a fundamental domain for the action of the group of isometries generated by the appropriate gluings of pairs of faces. The gluings that they inherit on their faces are exactly the gluings that yield the $5_2$ knot complement as in \cite{Adams1} or \cite{CDW}.

\begin{lemma} \label{Whiteheadfilling} If the $y$ and $v$ inequalities are equalities, a corresponding manifold has waist size  $\sqrt[4]{2}$  and the only such manifold is obtained  by  (2,1)-Dehn filling on one cusp of the Whitehead link.
\end{lemma}

Proof. One can check that the two inequalities are equalities exactly for $(w,e) = (\sqrt[4]{2},\sqrt[4]{2})$. This also forces the $m$-inequality to be an equality, meaning that the $1/m$-ball is a full-sized ball that is tangent to one of the pre-existing $1/w$-balls. Hence, we have two full-sized balls sharing a $1/w$-ball. These two full-sized balls must be associated, as if the $1/w$-ball they share is a 2 up, their opposite $1/w$-balls are both 2 down, and therefore identical up to the action of the cusp subgroup. This determines the cusp subgroup and shows that there is not enough room in the cusp diagram for a second pair of full-sized balls.  Thus the cusp diagram must appear as in 
Figure \ref{bigone}.

\begin{figure}{}
\begin{center}
\includegraphics[scale= .6]{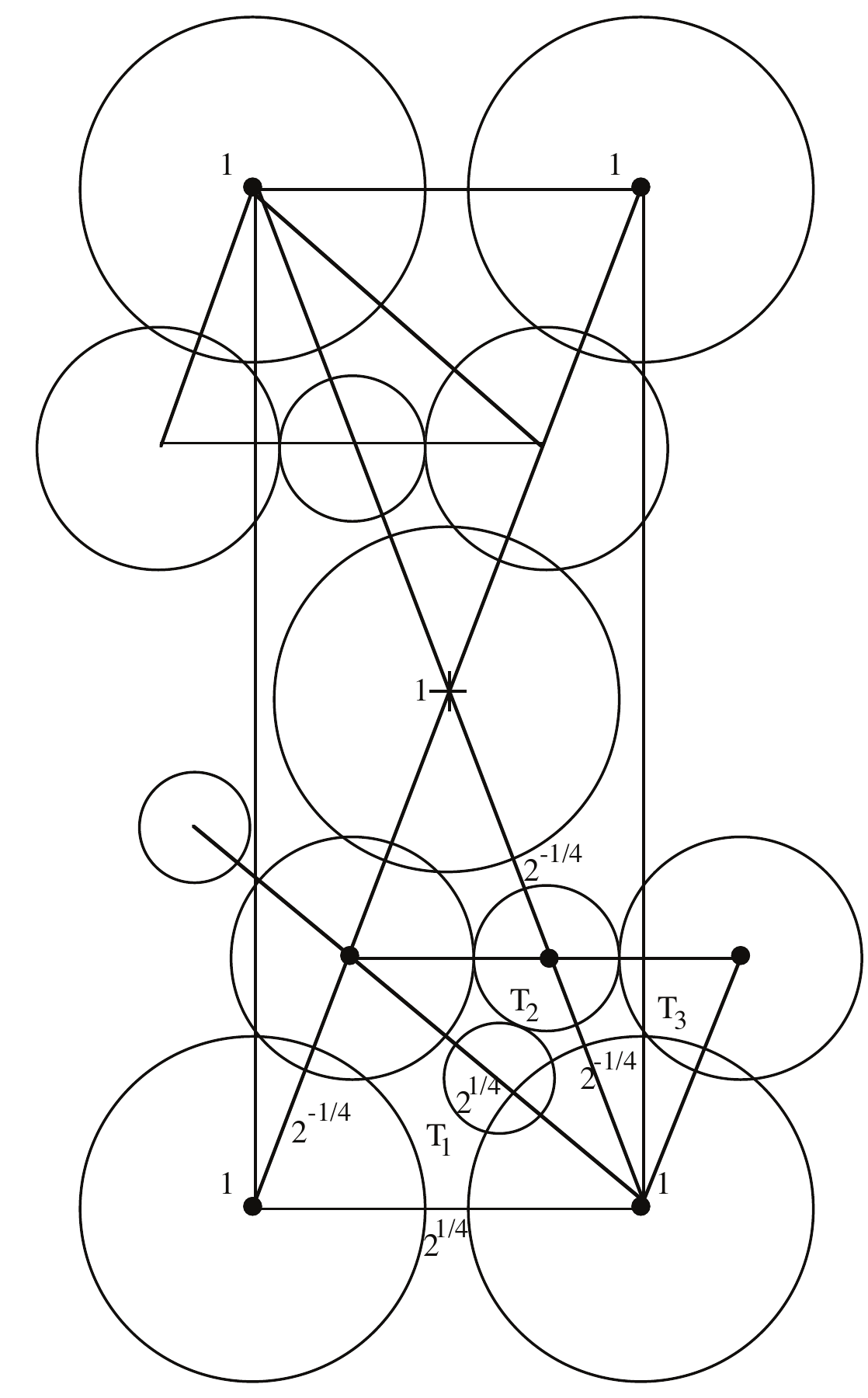}
\caption{Cusp diagram forced by $y$ and $v$ non-zero but minimal.}
\label{bigone}
\end{center}
\end{figure}

The isometries taking 1 up to 1 down and 2 up to 2 down ensure that three vertical tetrahedra $T_1, T_2$ and $T_3$ form a fundamental domain for the cusp diagram. Those same isometries give us the face identifications on those tetrahedra as in Figure \ref{identifications}.

\begin{figure}{}
\begin{center}
\includegraphics[scale= .6]{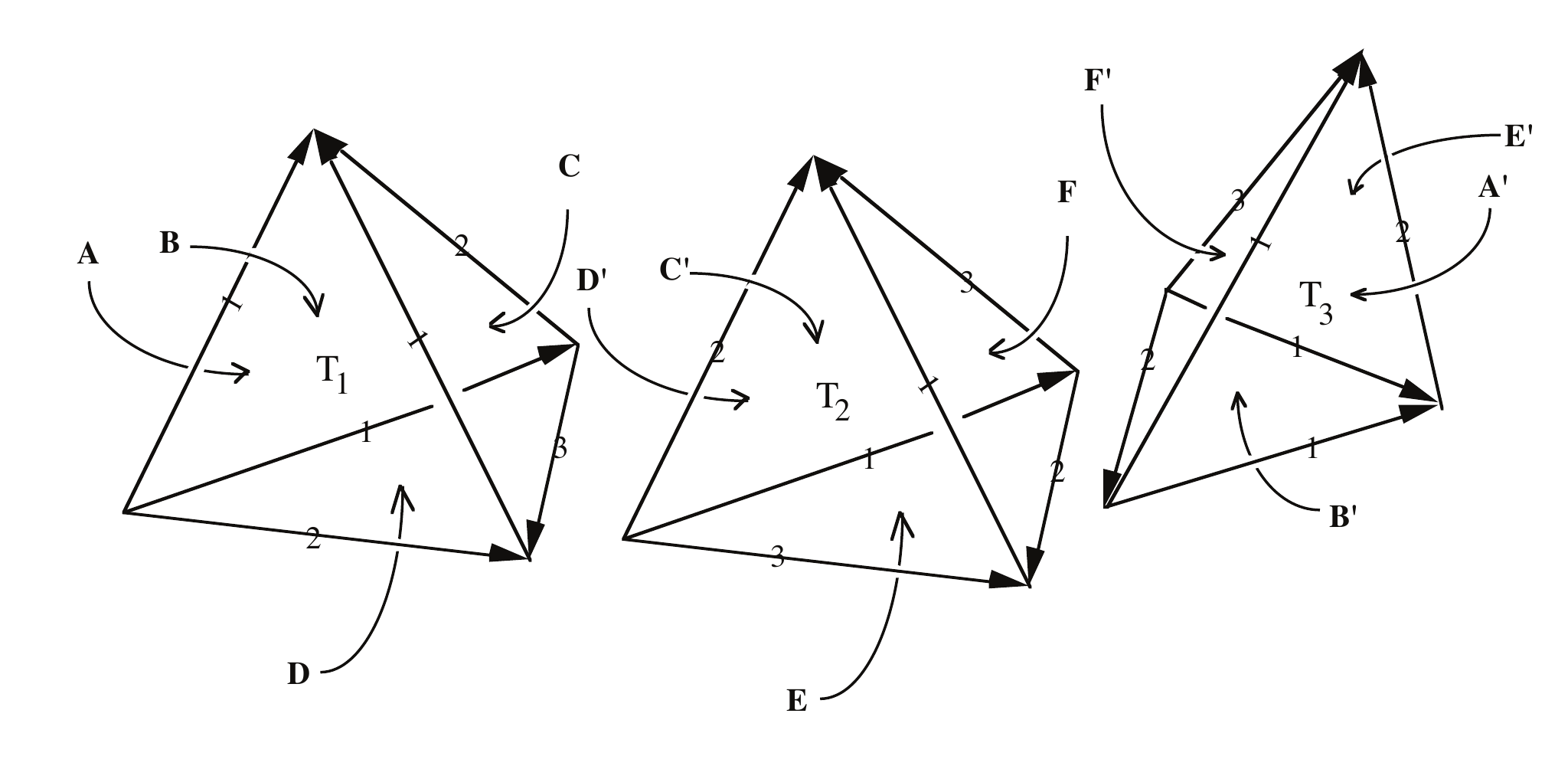}
\caption{Identifications on three ideal tetrahedra that yield m009.}
\label{identifications}
\end{center}
\end{figure}

The resulting manifold is m009 in the census of SnapPy (cf. \cite{CDW}), or (2,1)-Dehn surgery on the Whitehead link (or (-2,1)-surgery on its reflection.)  No other points can be identified in the interiors of the three tetrahedra, as if there were such identifications,  horoballs would overlap without being identical except in the case of identifying $T_1$ with $T_3$. However, if those two tetrahedra were identified, it would cause fixed points in $\mathbb{H}^3$ of nontrivial isometries in the fundamental group, a contradiction.

\begin{thm} The only manifolds of waist size at most  $\sqrt[4]{2}$ are the figure-eight knot complement, the $5_2$ knot complement and the manifold obtained by (2,1)-surgery on the Whitehead link.
\end{thm}

Proof. If $y = 0$, then the $1/e$-ball is a $1/w$-ball. Hence, $e = 1/e$, so $e = 1$ and two $1/w$-balls that are identified by $P$ are tangent to one another, forcing $w = 1$. Hence, by Theorem 3.1 of \cite{Adams3}, the resulting manifold must be the figure-eight knot complement. If $v = 0$, then the corresponding manifold must be either the figure-eight knot complement or the $5_2$ knot complement by Lemma \ref{knot5_2}. 

 \begin{figure}{}
\begin{center}
\includegraphics[scale= .6]{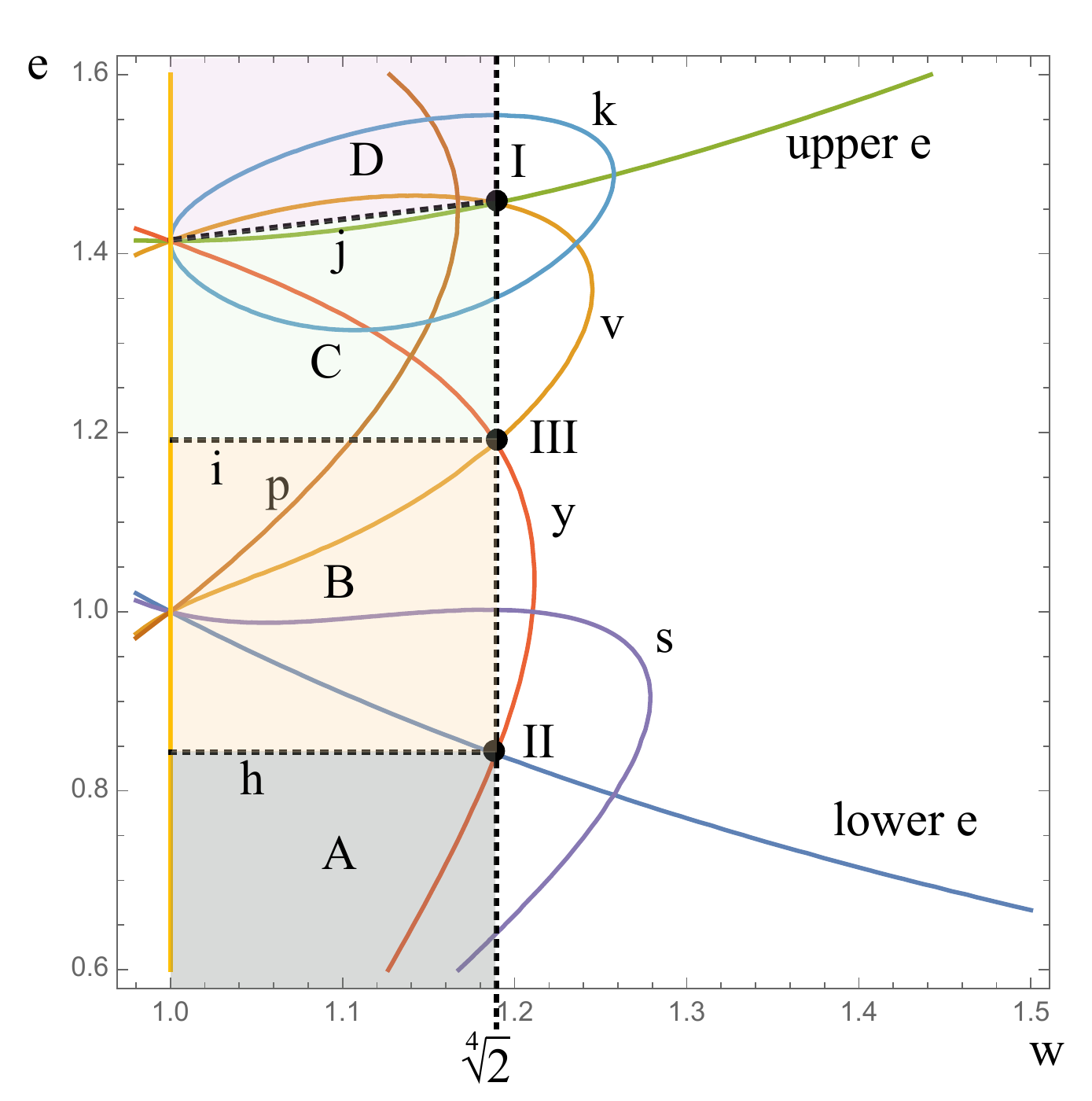}
\caption{Inequalities in the $we$-plane.}
\label{inequalities}
\end{center}
\end{figure}

      If both $y \neq 0$ and $v \neq 0$, then the inequalities from $y, v$ and the upper and lower bounds for $e$ apply. (See Figure \ref{inequalities}.) We prove, as can be seen in Figure \ref{inequalities}, that these four inequalities cover the region in the $we$-plane given by $1 \leq w < \sqrt[4]{2}$. 
      

We consider the four regions $A, B, C$ and $D$. The line segment $h$ given by $e = 1/\sqrt[4]{2}$ intersects the  lower $e$-curve and $y$-curve at one and the same point, when $w = \sqrt[4]{2}$.  Therefore rectangle $A$ is covered by the lower $e$-inequality. 

The line segment $i$ given by $e = \sqrt[4]{2}$ intersects both the $y$-curve and $v$-curve at one and the same point, when $w = \sqrt[4]{2}$. Therefore rectangle $B$ is covered by the $y$-inequality. 

Solving for the intersections of the two regions bounded by the upper $e$-equality and the $v$-equality, we find points $(1, \sqrt 2)$ and $(\sqrt[4]{2}, \frac{\sqrt{3}}{\sqrt[4]{2}})$. The line segment labelled $j$ is the line segment between these two points, forming part of the boundary of trapezoids $C$ and $D$.

To see that $j$ lies above the upper $e$-equality, we can solve $w^4 - e^2w^2 + 1 =0$ for $e$ and then show it is concave down as a function of $e$.  Hence, region $D$ is covered by the upper $e$-inequality.   

To see $j$ lies below the $v$-curve,  we use implicit differentation to compute $e_w$. Then we check that at $(1,\sqrt 2)$, $e_w = \frac{1}{\sqrt 2}$, which is greater than the slope of $j$. At $(\sqrt[4]{2}, \frac{\sqrt{3}}{\sqrt[4]{2}})$, $e_w$  is negative. Then using Mathematica, we check that there is only one place along the $v$-curve between these two points where $e_w$ equals the slope of $j$. Hence the $v$-curve must lie above the straight line segment $j$.  Hence trapezoid $C$ is covered by the $v$-inequality.   
         
  Thus, the  minimum possible waist size is then  $\sqrt[4]{2}$ , which occurs for three possible choices of $w$ and $e$.

 In the first case, $(w,e) = (\sqrt[4]2, \frac{\sqrt{3}}{\sqrt[4]{2}})$ (labelled $I$ in Figure \ref{inequalities}), the $k$-inequality is not satisfied, and hence the two horoballs at the ends of the segment labelled $k$ in Figure \ref{distances} must be identical. This choice of parameters then yields a cusp diagram as in Figure \ref{DiagramI}. This forces the centers of two $1/w^3$-balls corresponding to the same $1/w$-ball to have the same $x$-coordinates, and therefore one ball is sent to the other under the isometry $P$. However, the vertical edges coming out of their centers are identically labelled but oppositely oriented, a contradiction to the fact all isometries must be fixed point free.

 \begin{figure}{}
\begin{center}
\includegraphics[scale= .6]{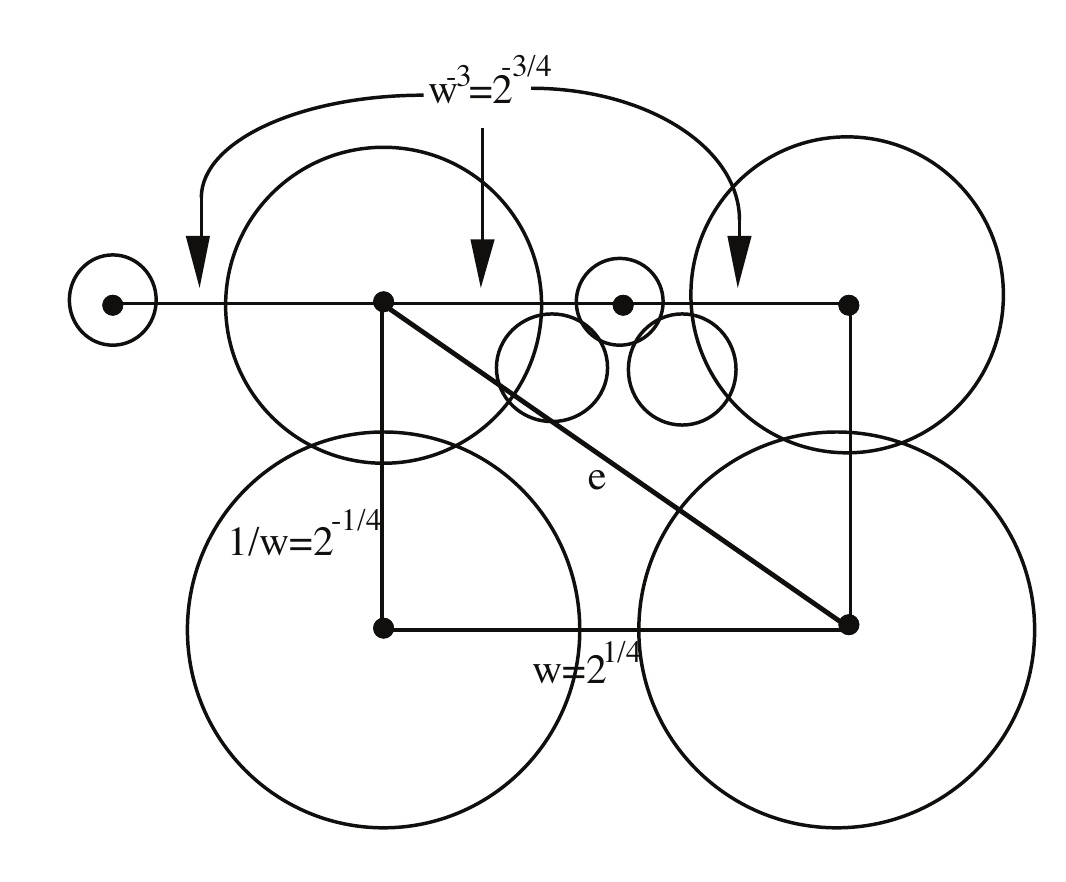}
\caption{Diagram corresponding to $I$.}
\label{DiagramI}
\end{center}
\end{figure}

     In the second case, $(w,e) = (\sqrt[4]2, \frac{1}{\sqrt[4]2})$ (labelled II in Figure \ref{inequalities}), we show that the ultimate result must also be an orbifold rather than a manifold. Since $e$ is as short as possible, the $1/e$-balls must be full-sized. Moreover, since the $s$-inequality is not satisfied, the two $1/w^3$-balls at the ends of the $s$ segment in Figure \ref{distances} must be identical. Since the $y$-inequality is exactly satisfied, the $1/w$-ball is tangent to the $1/e$-ball. Thus, one can see that our cusp diagram has two pieces, the first appearing  as in Figure \ref{pieceI}, and the second appearing the same except for the labels. Label the edge between the endpoints of two vertical down pointing 4 edges by a 6 edge pointed left to right. Applying the isometry that takes 4 down to 4 up, we see that the full-sized ball corresponding to 4 up has two $1/w$-balls labelled with 6 up and 6 down. These can only be fit in if the 6 edge coincides with the 2 edge. This forces the 4 edge to be the 1 edge and the diagram to appear as in Figure \ref{piece2}.

\begin{figure}{}
\begin{center}
\includegraphics[scale= .6]{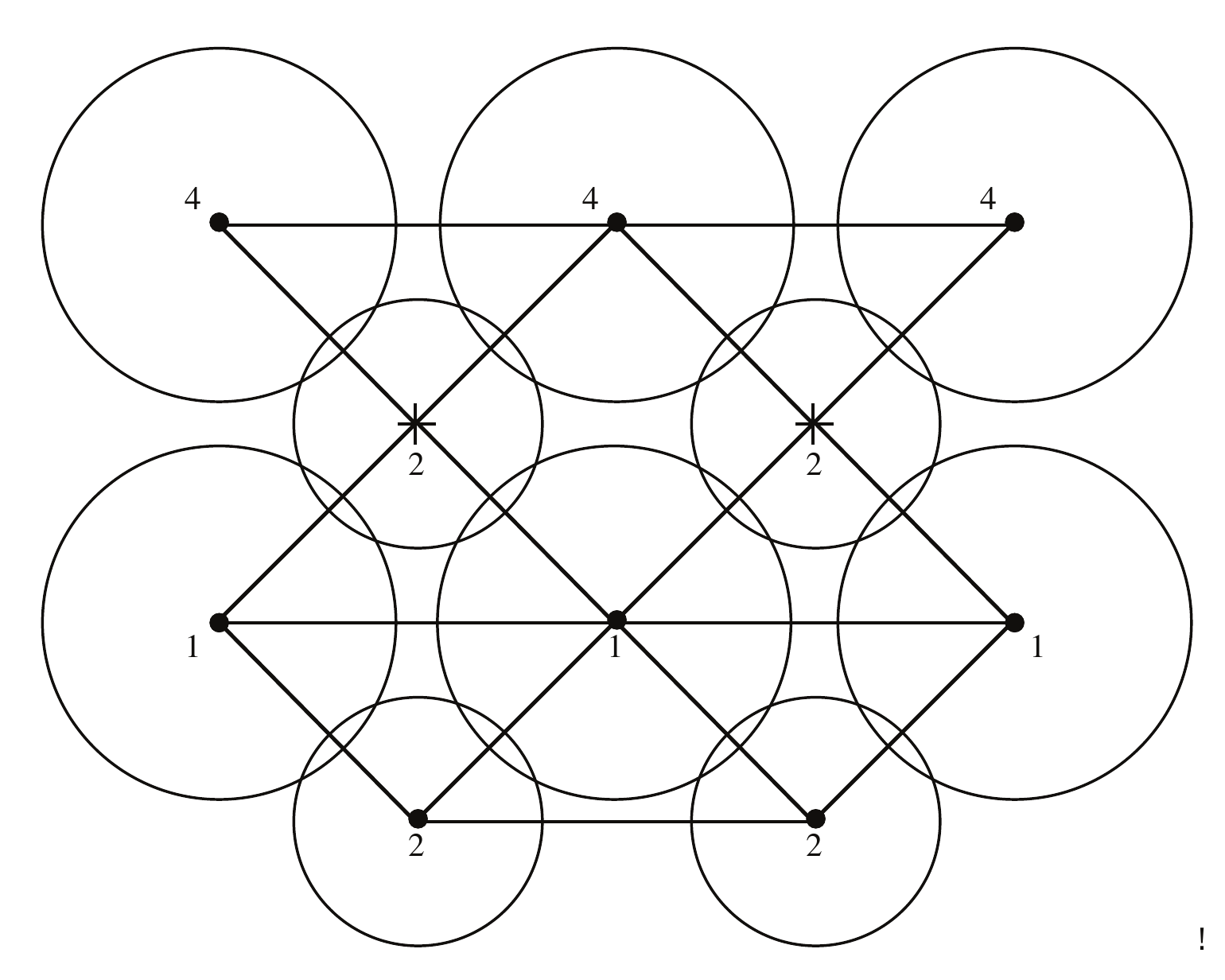}
\caption{One of two identical pieces except for the labelings for Case II.}
\label{pieceI}
\end{center}
\end{figure}

\begin{figure}{}
\begin{center}
\includegraphics[scale= .6]{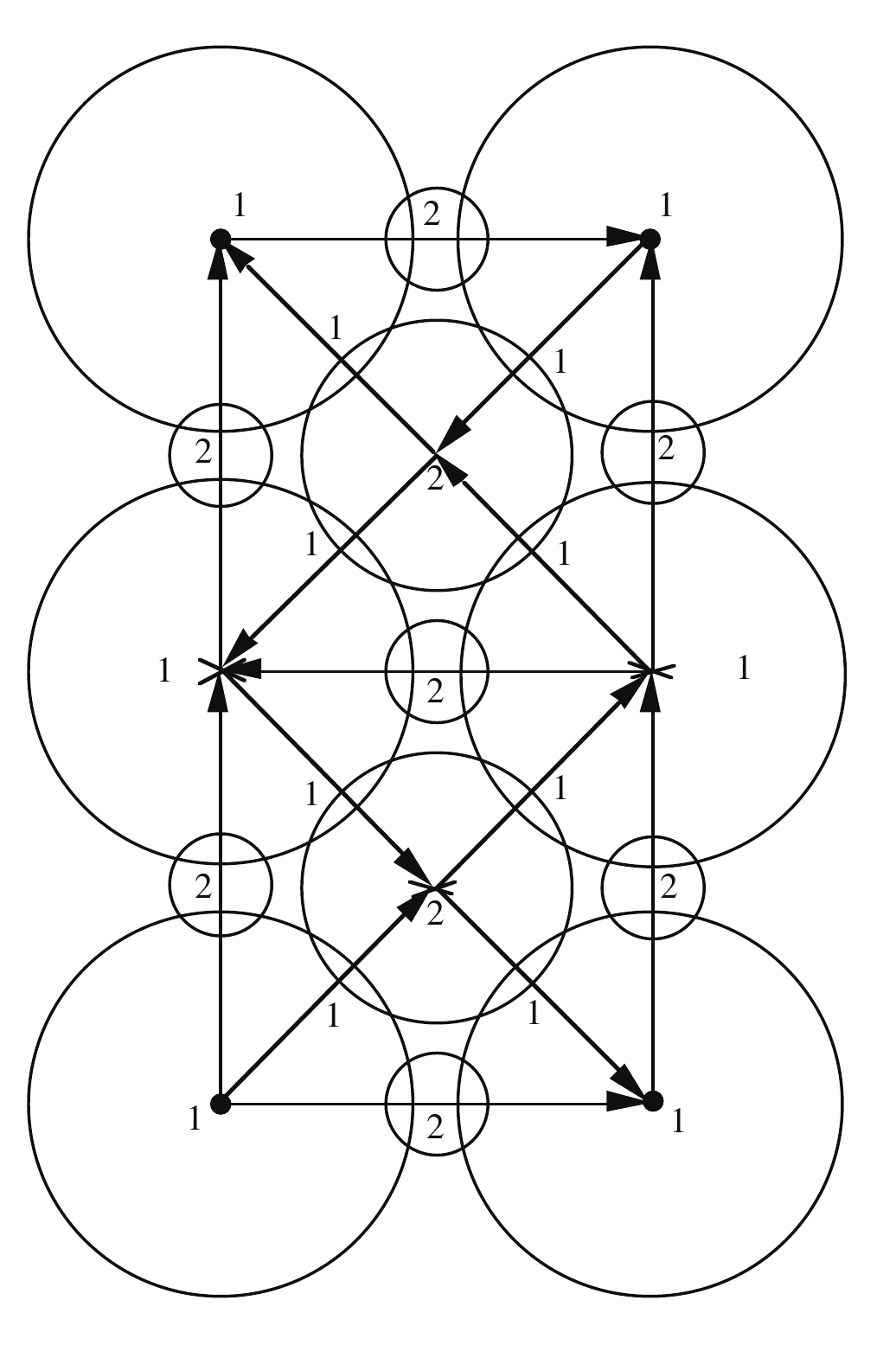}
\caption{Cusp diagram corresponding to II. }
\label{piece2}
\end{center}
\end{figure}

The labelings in Figure \ref{piece2} are forced by the isometry P and the isometries that take 1 up to 1 down and 2 up to 2 down.  Similarly to the arguments appearing previously in this paper, one sees that all eight of the vertical ideal tetrahedra are identified, meaning that the resulting quotient must be an orbifold or nonorientable manifold.

     The third case, $(w,e) = (\sqrt[4]2, \sqrt[4]2)$ (labelled III in Figure \ref{inequalities}), is the one corresponding to Lemma  \ref{Whiteheadfilling}, thereby showing that the manifold must be (2,1)-Dehn filling on one component of the Whitehead link.

\section{An Application}

     One application of the results on waist size is for the length of unknotting tunnels for 2-cusped hyperbolic manifolds. A cusped hyperbolic manifold is said to have tunnel number one if the corresponding compact manifold with toroidal boundaries contains a properly embedded arc (called an unknotting tunnel) such that the complement of a neighborhood of the arc is a genus two handlebody.  In \cite{Adams2}, it was shown that for a tunnel number one 2-cusped finite volume hyperbolic 3-manifold, any unknotting tunnel is isotopic to a vertical geodesic of length less than $\ln(4)$ where the length of the geodesic is measured for that part of it outside a choice of cusps with disjoint interiors.  In fact, Theorem 4.4 of that paper shows that the length of an unknotting tunnel is less than $\ln(4/w^2)$ where $w$ is the minimum of the current waist sizes for a choice of the two cusps with disjoint interiors.

\begin{corollary}An unknotting tunnel in a tunnel number one two-cusped hyperbolic 3-manifold is isotopic to a vertical geodesic of length less than $\ln(2^{7/4})$.
\end{corollary}

\begin{proof} Choose one cusp $C_1$ to be maximal with waist size $a_0$ greater than $2^{1/4}$ and the other cusp $C_2$ as large as possible without overlapping the first cusp on its interior.  The second cusp $C_2$ will either touch itself and therefore have waist size greater than $2^{1/4}$  or it will touch $C_1$ and have waist size $b_0$ at least 1. In the second case, expand $C_2$ while shrinking $C_1$ until they have same waist size. If $C_2$ becomes maximal before they have the same waist size, the minimum of the two waist sizes is larger than if they have the same waist size. So in all cases, the minimum for $w$ occurs when the two cusps have the same waist size. The expansion of $C_2$ and shrinking of $C_1$ can be parametrized by $h$, where the two cusps have the same waist size when $\frac{a_0}{h} = b_0h$.  Hence, at this time, $h = \sqrt{\frac{a_0}{b_0}}$ and both cusps have waist size  $\sqrt{a_0b_0}$. Thus,  $w$ is at least $2^{1/8}$, giving a universal bound for the length of unknotting tunnels of $\ln(2^{7/4})$.
\end{proof}

\end{document}